\documentclass[12pt]{amsart}

\usepackage{amsmath,amssymb,amscd,amsfonts,amsthm,verbatim}
\usepackage[all,cmtip]{xy}
\usepackage{paralist}
\usepackage[mathscr]{eucal}
\usepackage{color}
\usepackage{pdfsync}
\usepackage[]{fontenc}
\usepackage{enumerate}
\usepackage{pgfplots}
\usepackage{graphicx}
\usepackage{booktabs}
\usepackage{pdfpages}
\usepackage{tikz}
\usepackage{tikz-cd}
\usetikzlibrary{cd}
\usetikzlibrary{through}
\usetikzlibrary{decorations.pathreplacing,decorations.markings}
\usetikzlibrary{matrix,arrows,decorations.pathmorphing}

\pgfplotsset{width=7cm,compat=1.3}
\usepackage[pdftex]{hyperref}
\hypersetup{citecolor=blue,linktocpage}

\usepackage[colorinlistoftodos]{todonotes}
\newcommand\lsi[1]{\todo[inline,color=pink!60]{#1}} 

\newcommand\rp[1]{\todo[color=yellow]{#1}}
\newcommand\rpi[1]{\todo[inline,color=yellow]{#1}}

\newtheorem{thm}{Theorem}[section]

\newtheorem{prop}[thm]{Proposition}

\newtheorem{cor}[thm]{Corollary}

\newtheorem{assu-nota}[thm]{Assumption--Notation}

\newtheorem*{thm*}{Theorem}

\theoremstyle{definition}

\newtheorem{rem}[thm]{Remark}

\newtheorem{qst}[thm]{Question}

\newcommand{\Q}{\mathbb Q}
\newcommand{\R}{\mathbb R}

\newcommand{\N}{\mathbb N}

\newcommand{\ud}{^{(d)}}
\newcommand{\up}[1]{^{(#1)}}

\DeclareMathOperator{\Aut}{Aut}
\DeclareMathOperator{\Pic}{Pic}
\DeclareMathOperator{\pic}{Pic}

\DeclareMathOperator{\vol}{vol}

\newcommand{\fie}{\varphi}

\newcommand{\wt}{\widetilde}

\numberwithin{equation}{section}

\title{Higher dimensional Clifford-Severi equalities}
\author{Miguel \'Angel Barja}
\author{Rita Pardini}
\author{Lidia Stoppino}
\address{Miguel \'Angel Barja\\Departament de Matem\`atiques\\Universitat Polit\`ecnica de Catalunya\\Avda. Diagonal 647\\08028 Barcelona\\Spain}
\email{miguel.angel.barja@upc.edu}
\address{Rita Pardini\\Dipartimento di Matematica\\Universit\`a di Pisa \\Largo B. Pontecorvo 5\\I-56127  Pisa\\Italy}
\email{rita.pardini@unipi.it}
\address{Lidia Stoppino \\Dipartimento di Scienza e Alta Tecnologia\\Universit\`a dell'Insubria \\ Via Valleggio 11\\22100, Como\\Italy}
\email{lidia.stoppino@uninsubria.it}

\thanks{The first author was supported by MINECO MTM2015-69135-P ``Geometr{\'\i}a y Topolog{\'\i}a de Variedades, \'Algebra y Aplicaciones" and by Generalitat de Catalunya 2017SGR932. The second and third authors are members of G.N.S.A.G.A.--I.N.d.A.M.  This research was partially supported  and by MIUR (Italy) through  PRIN 2015 ``Geometria delle variet\`a algebriche" and PRIN 2012-13 ``Moduli, strutture geometriche e loro applicazioni''.}

\begin{document}
\begin{abstract}
Let $X$ be a  smooth  complex projective variety,  $a\colon X\rightarrow A$ a morphism  to an abelian variety such that $\pic^0(A)$ injects into $\pic^0(X)$  and let
$L$ be a line bundle on $X$; denote by $h^0_a(X,L)$ the minimum  of $h^0(X,L\otimes a^*\alpha)$ for $\alpha\in \Pic^0(A)$.
 The  so-called Clifford-Severi inequalities have been proven in \cite{barja-severi} and \cite{BPS1}; in particular,  for any $L$ there  is  a lower bound  for the volume given by:
 $$\vol(L)\ge n! h^0_a(X,L),$$
 and, if $K_X-L$ is pseudoeffective,
 $$\vol(L)\ge 2n! h^0_a(X,L).$$
In this paper  we characterize varieties and line bundles  for which the above  Clifford-Severi inequalities are equalities.

\end{abstract}
\maketitle
\setcounter{tocdepth}{1}
\tableofcontents
\section{Introduction and statement of results}

A minimal surface $S$ of general type and of maximal Albanese dimension satisfies the Severi inequality $K_S^2\geq 4\chi(K_S)$ (\cite{rita-severi}).
The history  of this result spans more than 70 years,  from the original (incorrect) proof of Severi \cite{severi} to the complete proof given in \cite{rita-severi} by the second named author.
Then  the first named author observed that,  thanks to  the generic vanishing theorem,   the Severi  inequality can be interpreted as a lower bound for the ratio of the volume of $K_S$ to the general value of $h^0(S,K_S\otimes \alpha)$ for $\alpha \in \Pic^0(S)$, and that it makes sense to look for lower bounds of this type also  for line bundles other than the canonical  bundle.
So in  \cite{barja-severi} he considered triples $(X, a,L)$, where $X$ is a smooth projective variety of dimension $n$, $a\colon X\to A$ is a morphism such that $\dim a(X)=n$  and  $a^*\colon \Pic^0(A)\to\Pic^0(X)$ is injective, and $L$ a nef line bundle on $X$. He defined the continuous rank $h^0_a(X,L)$ as the generic value of $h^0(X, L\otimes a^*\alpha)$ for $\alpha\in\Pic^0(A)$ and gave a set of ``Clifford--Severi inequalities'' in arbitrary dimension, i.e. lower bounds for the {\em slope} $\lambda(L):=\frac{\vol(L)}{h^0_a(X,L)}$.
The name of Clifford was added because in dimension 1 these inequalities are just a continuous version of the classical Clifford's Lemma for line bundles on a curve.

Finally, in \cite{BPS}  we  strengthened  and refined significantly  the  lower bounds for $\lambda(L)$  given by the Clifford--Severi inequalities. At the same time we  substantially streamlined the proofs.
This was made possible by the introduction of two new tools, the continuous rank function and the eventual map, that  are also  essential for proving the results of this paper.

\smallskip

The most relevant  Clifford-Severi inequalities are the following:
\begin{equation}\label{eq: severi1}
  \vol(L)\ge n! h^0_a(X,L),
  \end{equation}
   and, if $K_X-L$ is pseudoeffective,
   \begin{equation} \label{eq: severi2}
   \vol(L)\ge 2n! h^0_a(X,L).
   \end{equation}
Note that the Severi inequality for surfaces is a special case of \eqref{eq: severi2}.

When $h^0_a(X,L)>0$, inequalities \eqref{eq: severi1} and \eqref{eq: severi2} can be rephrased as $\lambda(L)\geq n!$ and $\lambda(L)\geq2n!$, respectively.
We say that a line bundle $L$ belongs to the first (respectively,  second) {\it Clifford-Severi line} iff $h^0_a(X,L)>0$ and $\lambda(L)= n!$ (respectively, $\lambda(L)=2n!$).

Both  inequalities (\ref{eq: severi1})  and  (\ref{eq: severi2}) are sharp: ample line bundles $L$ on abelian varieties $X=A$ clearly belong to the first Clifford Severi line;
in addition, if   $a\colon X\rightarrow A$ is a double cover of an abelian variety branched on  a smooth  ample divisor  $B\in |2R|$, any line bundle $L=a^*N$ with $0\leq N \leq R$  verifies $\lambda(L)=2n!$, i.e. $L$   belongs to the Second Clifford-Severi line.

Still,  the classification of these extremal  cases has proven surprisingly hard, due to the fact that the proof of the inequalities in \cite{rita-severi} and in \cite{barja-severi} is based on a limiting argument. For example, the classification of surfaces on the Severi line $K^2_S=4\chi(K_S)$  was achieved only in 2015 (\cite{BPS} and  \cite{LZ2}, independently), ten years later than \cite{rita-severi}. The higher dimensional cases and  the case of a general $L$ (even for surfaces) were completely open.
\smallskip

In this paper, we classify the triples $(X,L,a)$
achieving equality in  \eqref{eq: severi1} and \eqref{eq: severi2};
the   classification essentially reduces to the known examples described above. Let us state our main results.

Given a morphism  $a\colon X\to A$ of a $n$-dimensional variety $X$ to an abelian variety $A$, we say that $a$ is {\em strongly generating}  if   $a^*\colon \Pic^0(A)\to\Pic^0(X)$ is injective, and {\em of maximal $a$-dimension} if $\dim a(X)=n$.

First we  give a characterization of pairs   $(X,L)$ on the  first Severi line:
\begin{thm}\label{thm: Severi line 1}
Let $X$ be a smooth variety of dimension $n\geq 1$, let $a\colon X\to A$ be a strongly generating map to a $q$-dimensional abelian variety such that $X$ is of maximal $a$-dimension;  let $L\in {\Pic}(X)$ be a line bundle with $h^0_a(X,L)>0$.  \par
If $\lambda(L)=n!$, then  $q=n $ and $a\colon X\to A$ is a birational morphism.
\end{thm}

Next we consider pairs $(X,L)$ on the second Clifford-Severi line. 

If the assumptions of Theorem \ref{thm: Severi line 1} are satisfied, one can define the {\em eventual map} given by $L$ (\cite{BPS1}, cf. \S \ref{ssec: eventual}): this is a generically finite map whose degree we denote by $m_L$.
For $n\ge 2$,  we have  the following characterization:
\begin{thm}\label{thm: Severi line 2} Let $X$ be a smooth variety of dimension $n\geq  2$, let $a\colon X\to A$ be a strongly generating map to a $q$-dimensional abelian variety  such that $X$ is of maximal $a$-dimension;  let $L\in {\Pic}(X)$ be a line bundle with $h^0_a(X,L)>0$.  \par
  If $K_X-L$ is pseudoeffective  and $\lambda(L)=2n!$, then:
\begin{enumerate}
\item $q=n$, $\deg a=2$;
\item if $\chi(K_X)>0$,  then $K_X=a^*N+E$ for some ample $N$ in $A$ and  some (possibly non effective) $a$-exceptional divisor $E$, $m_{K_X}=2$ and $a$ is birationally equivalent to the eventual map of $K_X$;

\item if in addition $K_X-L+\alpha$ is effective for some  $\alpha\in \pic^0(A)$, then $m_L=2$ and $a$ is  birationally equivalent to the eventual map of $L$.
\end{enumerate}
\end{thm}
The case of curves is dealt with separately (Theorem \ref{thm: curve}). It is worth observing that the above results are new in any dimension.

Theorem \ref{thm: Severi line 2} gives the following characterization of varieties on the second Clifford-Severi line for $L=K_X$, extending the analogous result for surfaces (\cite{BPS}, \cite{LZ2}).
\begin{cor}\label{cor: Severi Canonico}
Let $X$ be a complex projective minimal $\mathbb{Q}$-Gorenstein variety of dimension $n\ge 2$ with terminal singularities,  let $a\colon X\rightarrow A $  be the Albanese map and let $\omega_X={\mathcal O}_X(K_X)$ be the canonical  sheaf. \par
If $X$ is of maximal Albanese dimension and $K_X^n= 2 n!\chi(\omega_{X})>0$, then:
\begin{enumerate}
\item $q=n$ and $\deg a=2$;
\item if $X'\to X$ is a desingularization and $a'\colon X'\to A$ is the Albanese map, then $K_{X'}={a'}^*N+E$ for some ample $N$ in $A$ and   $a'$-exceptional divisor $E$, $m_{K_{X'}}=2$ and $a'$ is birationally equivalent to the eventual map of $K_{X'}$.
\end{enumerate}
\end{cor}


The key point in our proofs are the new  techniques introduced in \cite{BPS1},  i.e. the construction of  the eventual map 
associated to a line bundle such that $h^0_a(X,L)>0$ (see \ref{ssec: eventual}), and the {\em continuous rank function} (see \ref{ssec: continuous rank}):
given any ample line bundle $H $ on $A$ there is a natural way to  define a real valued function  $\phi(x)=h^0_a(X, L+xa^*H)$ of  $x\in \R$, which  is convex and continuous.
With this new approach the limiting argument of \cite{rita-severi} and \cite{barja-severi} is embodied  in the definition of the rank function.
Moreover, the Clifford-Severi inequalities imply a relation between the rank function and the volume function $\vol_X(L+xa^*H)$ (\cite{LM}), i.e. the non-negativity of the functions $\vol_X(L+xa^*H)-n!h^0_a(X, L+xa^*H) $ and $\vol_X(L+xa^*H)-2n!h^0_a(X, L+xa^*H) $, respectively. We  approach the classification problem by studying the minima of these two functions.
\smallskip

When $L=K_X$,  we have  proven  (Theorem \ref{thm: Severi line 2} (ii))  that $K_X$ is the pullback of an ample divisor on $A$ plus an $a$-exceptional divisor. In the case of surfaces we can  say  more (cf. \cite{BPS}):  the $a$-exceptional divisor is effective. This is equivalent to the fact that, given the Stein factorization of $a$: $X\to \overline X\to A$, the surface $\overline X$ has canonical singularities.
\begin{qst}\label{conj: Severi1}
In the hypoteses of Theorem \ref{thm: Severi line 1}, or of   \ref{thm: Severi line 2}, is it true that $L=a^*N+E$ for some ample $N$ in $A$ and  some {\it effective} $a$-exceptional divisor $E$?
\end{qst}


\section{Preliminary results}\label{sec: Preliminaries}
\subsection{Notation and conventions}
We work over the complex numbers; varieties and subvarieties are assumed to be irreducible and projective.
Since our  focus is on birational geometry,  a   {\em map}  is a   rational map and we denote all maps by solid arrows.
 We say that two dominant maps $f\colon X\to Z$, $f'\colon X\to Z'$ are {\em  birationally equivalent} if there exists a birational isomorphism $h\colon Z\to Z'$ such that $f'=h\circ f$.

 We do not distinguish between divisors and line bundles and use both the additive and multiplicative notation, depending on the situation.

 If $d$ is a non-negative integer, we write ``$d\gg 0$'' instead of ``$d$ large and divisible enough''.
\bigskip

We recall briefly some definitions and results from  \cite{BPS1},  Sections 3 and 4. We refer the reader to \cite{BPS1}  for more details.
\subsection{The continuous rank}\label{ssec: continuous rank}
The starting point is a smooth variety $X$ of dimension $n$ with a map $a\colon X \rightarrow A$ to an abelian variety of dimension $q$. We  say that $a$ is {\em strongly generating} if $a^*\colon \Pic^0(A)\to\Pic^0(X)$  is injective and we  say that $X$ is {\em of maximal $a$-dimension} if $\dim  a(X)=n$.
We assume throughout all the section  that  both these conditions hold for $a\colon X\to A$; to simplify the notation,  given  $\alpha\in \Pic^0(A)$    we write   $\alpha$ instead of  $a^*\alpha$.
Given a line bundle $L\in{\rm Pic}(X)$, we define its {\it continuous rank} (with respect to $a$)  as the minimum $h^0_a(X,L)$  of $h^0(X,L\otimes  \alpha)$ for $\alpha \in \Pic^0(A)$.

Consider  the  cartesian  diagram:
\begin{equation}\label{diag:-1}
\begin{CD}
X\up{d}@>\wt{\mu_d}>>X\\
@V{a_d}VV @VV a V\\
A@>{\mu_d}>>A
\end{CD}
\end{equation}
where $\mu_d$ denotes multiplication by $d$. The map $\wt \mu_d$ is \'etale of degree $d^{2q}$ and $X\ud$ is connected  since  $a$ is strongly generating; we set $L\ud:=\wt \mu_d^* L$.

One of the reasons for  introducing the continuous rank is the following {\em multiplicative property}:
\begin{equation}\label{eq: multiplication}
h^0_{a_d}(X\ud, L\ud)=d^{2q}h^0_a(X,L)\,\,\, \mbox{ for any }d\in \N.
\end{equation}
Now fix a very ample line bundle $H\in \Pic^0(A)$ and set $M:=a^*H$, and, more generally, $M_d:=a_d^*H$ for all $d\in \N$. The line bundles $M\ud$ and $M_d$ on $X\ud$ are related by the formula:
\begin{equation}\label{eq: division}
M\ud =d^{2q}M_d\quad \mod \Pic^0(A).
\end{equation}

For $x\in \R$ consider the $\R$-line bundle $L_x:=L+xM$; if $x$ is rational, then by \eqref{eq: division} we have that  $(L_x)\ud=L\ud +xM\ud=L\ud+xd^2M_d$ is an integral line bundle if $d\gg 0$ and $h^0_a(X,L_x):=\frac{1}{d^{2q}}h^0_{a_d}(X\ud, L\ud+xd^2M_d)$ is well defined thanks to \eqref{eq: multiplication}.

The function $x\mapsto h^0_a(X,L_x)$ just defined for $x\in \Q$ can be extended to a continuous, convex function $\phi(x)$ on $\mathbb{R}$.  In particular, $\phi$   has one-sided derivatives at every point $x\in \mathbb{R}$ and is differentiable except at most at countably many points.

Given a subvariety $T\subseteq X$, we  define   the {\em restricted continuous rank} $h^0_a(X_{|T}, L)$ as the generic value of $\dim {\rm Im}(H^0(X,L\otimes \alpha)\rightarrow H^0(T,(L\otimes \alpha)_{|T}))$, for $\alpha \in \Pic^0(A)$.  The restricted continuous rank also  satisfies the multiplicative property
\eqref{eq: multiplication}, so  for rational values of $x$ one can define  $h^0(X_{|T},L_x)$ as above,  and also in this case the function $h^0(X_{|T},L_x)$ extends to a continuous function of $x\in \R$.
Using this definition we are able  to give an explicit formula for the left derivative of $\phi$:
\begin{equation}\label{eq: derivata phi}
D^-\phi(x)=\lim_{d\rightarrow \infty}\frac{1}{d^{2q-2}}h^0_{a_d}(X\up{d}_{|M_d},(L_x)^{(d)}), \quad \forall x\in\R.
\end{equation}
\subsection{Slope and degree of subcanonicity}\label{ssec: subcanonicity}
Assume that $h^0_a(X, L)>0$ and define the slope of $L$ as:
$$\lambda(L):=\frac{\vol(L)}{h^0_a(X,L)}.$$
Note that  by  \cite[Proposition 3.2]{BPS1} the line bundle $L$ is big, and so  $\lambda(L)>0$.

The Clifford-Severi inequalities give explicit lower bounds for $\lambda(L)$, which involve the dimension of $X$ and the so called {\em numerical degree of subcanonicity} of $L$ (with respect to $M$). This is defined as:
\begin{equation}
r(L,M):=\frac{LM^{n-1}}{K_X M^{n-1}}\in (0,\infty].
\end{equation}
If  $r(L,M)=\infty$, namely if $K_X M^{n-1}=0$ then $\kappa(X)\le 0$, because $M$ is  big.  On the other hand, we have $\kappa(X)\ge 0$ since $X$  is of maximal $a$-dimension, so we conclude that $r(L,M)=\infty$ implies $\kappa(X)=0$.

\subsection{The eventual map and the eventual degree} \label{ssec: eventual}
Assume that $h^0_a(X,L)>0$. One of the main results of \cite{BPS1} (Theorem 3.7) is the definition of the {\em eventual map} associated with $L$: it is a generically finite dominant map $\varphi\colon X\to Z$, characterized up to birational equivalence by the following properties:
\begin{itemize}
\item[(a)] $a\colon X\to A$ factorizes as $X\overset{\fie}{\to}Z\overset{g}{\to} A$ for some map $g$.
\item[(b)] By property (a), we have a cartesian diagram:

 \begin{tikzcd}[ampersand replacement=\&]
X\up{d} \arrow[rightarrow]{r}{\fie\up{d}} \arrow{d}{\widetilde{\mu_d}} \& Z\up{d} \arrow{d}\arrow[rightarrow]{r} \& A\arrow{d}{\mu_d}\\
X \arrow[rightarrow]{r}{\fie} \& Z \arrow[rightarrow]{r}{g} \& A
\end{tikzcd}

Then the map $\fie\ud$ is birationally equivalent to the map given by $|L\ud\otimes \alpha|$ for $d\gg 0$ and $\alpha\in \Pic^0(A)$ general.
\end{itemize}
The degree of $\fie$ is denoted by $m_L$ and is called the {\em eventual degree} of $L$.
It is immediate to see that $L$ and $L\ud$ have the same eventual degree, so for $x\in \Q$ we choose $d$ such that ${L_x}\ud$ is integral and we set $m_{L_x}:=m_{(L_x)\ud}$. The function $x\mapsto m_{L_x}$  is non increasing and takes  integer values, so it can be extended to a left continuous function of $x\in \R$.
\smallskip


\subsection{Continuous resolution of the  base locus}\label{ssec: continuous-res}
Let  $\sigma\colon \wt X\to X$ be a birational morphism, set  $\wt L:=\sigma^*L$ and denote by $\wt a\colon \wt X\to A$ the map induced by $a$; we have $h^0_{\wt a}(\wt X, \wt L)=h^0_a(X,L)$, $\vol(\wt L)=\vol(L)$ and   the eventual map  $\wt \fie$ given by $\wt L$ is equal to $\fie\circ \sigma$ (in particular, $m_{\wt L}=m_L$).  So, for many purposes, we may replace $(X,L)$ by $(\wt X,\wt L)$.

Since it is  easier to deal with morphisms than with rational maps, the following construction ({\em continuous resolution of the  base locus}) is  often useful.
 It is possible to choose the morphism $\sigma\colon \wt X \to X$ in such a way that  there is a decomposition $\wt L=P+D$  as a sum of effective divisors,  such that for $d\gg 0$:
  \begin{itemize}
  \item[(a)]  the system  $|P\up{d}\otimes \alpha|$ is free for all $\alpha \in \pic^0(A)$;
  \item[(b)] $|P\up{d}\otimes \alpha|$ is  the moving part of  $|\wt L\ud\otimes \alpha|$  for $\alpha \in \Pic^0(A)$ general.
  \end{itemize}

 Clearly, one has $h^0_{\wt a}(\wt X, P)=h^0_{\wt a}(\wt X, \wt L)=h^0_{a_d} (X\ud,L\ud)=d^{2q}h^0_a(X,L)$, where $\wt a\colon \wt X\to A$ is the morphism  induced by $a$. Note that $P$ and $D$ are divisors on $\wt X$ and that the modification $\sigma \colon \wt X\to X$ satifying the above properties is not unique; however we will usually assume that a suitable $\sigma$ has been chosen and we will call $P$  and $D$ the {\em continuous moving part}  and the {\em continuous fixed part} of $L$, respectively. Actually, given  $x\in \Q$  such that $h^0_a(X,L_x)>0$ it is convenient to be able to speak about the continuous moving part $P_x$ of $L_x$ and its continuous fixed part $D_x$.  In order to do this we choose $d\in \N$ such that $(L_x)\ud$ is integral and then choose a modification $\eta\colon Y\to X\ud$ on which  $\eta^*\left((L_x)\ud\right)$ decomposes as a sum  $P+ D$ of its continuous moving and fixed part. One would like to have  divisors $P_x$ and $D_x$ on (a modification of)  $X$ that pull back to $P$ and $D$ on $Y$ and call these the continuous moving and fixed part of $L_x$, but such $P_x$ and $D_x$ usually do not exist. Moreover   the construction of $P$ and $D$ involves   several choices; however,   some of the  invariants attached to the  decomposition do not depend on these choices and so it makes sense to define them, in particular we can define:
 \begin{itemize}
 \item $h^0_a(X,P_x):=\frac{1}{d^{2q}}h^0_{\alpha }((Y, P)=\frac{1}{d^{2q}}h^0_a(X,(L_x)\ud)$, where $\alpha\colon Y\to A$ is the map induced by $a_d$;
 \item $(P_x)^{n-1}M:= \frac{1}{d^{2q}}P^{n-1}(\eta^*M)$;
 \item $m_{P_x}=m_P=m_L$.
 \end{itemize}
Moreover, we will sometimes write $(P_x)\ud$,  implying that we have chosen $d$ and $\eta\colon Y\to X\ud$ as above and we are considering the continuous moving part $P$ of $\eta^*\left((L_x)\ud\right)$.

\subsection{The volume function}\label{ssec: volume}

We refer the reader to \cite{lazarsfeld-I}, \cite[\S~2.2.C]{lazarsfeld-II} and \cite{elmnp} for  a complete  account of the properties of the volume and of the restricted volume. Here we just recall what we need in our  proofs.
\smallskip

The  {\em volume}  of a line bundle $L\in\Pic(X)$ is defined as:
$$\vol(L)=\vol_X(L):=\limsup_m\frac{n!h^0(X,mL)}{m^n}$$
and, more generally, given a subvariety  $T\subset X$ the {\em restricted volume} is defined as:
$$\vol_{X|T}(L):=\limsup_m\frac{n!h^0(X_{|T},mL)}{m^n}.$$

For $t\in \N$ one has  $\vol(tL)=t^n\vol(L)$ and $\vol_{X|T}(tL)=t^k\vol_{X|T}(L)$, where $k=\dim T$,  so both definitions generalize naturally  to $\Q$-line bundles.
Below  we state  in our setting some useful  consequences  of the properties of the volume of a $\Q$-line bundle (cf. also \cite{BPS1}). We assume $h^0_a(X,L)>0$ and we denote by $M$  a general element of $|M|=|a^*H|$ and by $M_d$ a general element of  $|M_d|=|a_d^*H|$.
\begin{itemize}
\item[(i)] if $L=P+D$, with  $D$ effective, then $\vol(L)\ge \vol(P)$ and   $\vol_{X|M}(L)\ge \vol_{X|M}(P)$;
\item[(ii)] if $\eta\colon \wt X\to X$ is a birational morphism, then $\vol_{\wt X}(\eta^*L)=\vol_X(L)$ and $\vol_{\wt X|\wt M}(L)=\vol_{X|M}(L)$, where $\wt M:=\eta^*M$;
\item[(iii)] $\vol_{X\ud}(L\ud)=d^{2q}\vol_X(L)$ and $\vol_{X\ud|M\ud}(L\ud)=d^{2q}\vol_{X|M}(L)$  for $M\in |M|$ general;
\item[(iv)] if $L$ is nef then $\vol_X(L)=L^n$ and $\vol_{X|M}(L)=L^{n-1}M=\vol_M(L_{|T})$.
\end{itemize}
\smallskip

 The  definition  of  volume  can  be extended to $\R$-line bundles, giving  a continuous function  on $N^1(X,\R)$.
 We  consider  the continuous function $\psi(x): =\vol(L_x)$, for $x\in \R$:  if $\vol(L_x)>0$ (i.e., if $L_x$ is big), $\psi$ is  differentiable  and we have (cf Thm.~A and  Cor.~C of   \cite{BFJ},  and  also \cite[Cor. C]{LM}):
  \begin{equation}\label{eq: derivata psi}
\psi'(x)=n\vol_{X|M}(L_x),
\end{equation}
where  $\vol_{X|M}$ is the restricted volume.

We close  this section  with a remark that  allows us to compare  the derivatives of $\psi(x)$ and of the continuous rank function $\phi(x)$.
This will be a key point of our arguments.

\begin{rem}\label{rem: Preliminaries}
Fix $x\in \Q$  such that $h^0_a(X, L_x)>0$ and   consider the continuous moving and fixed part   $P_x$ and $D_x$ of $L_x$ as in Section \ref{ssec: continuous-res}.

Thanks to the above properties (ii) and (iii)  of the volume, 
we can consider:
  $$\vol_{X|M}(P_x):=\frac{1}{d^{2q}}\vol_{X\ud|M\ud}\left((P_x)\ud\right),$$
  where $d\gg 0$.
 By property (iv) we have: 
\begin{gather}\label{eq: eq1}
\vol_{X|M}(P_x)=\frac{1}{d^{2q}}\left((P_x)\ud\right)^{n-1}M\ud=\frac{1}{d^{2q-2}}\left((P_x)\ud\right)^{n-1}M_d=\\
=\frac{1}{d^{2q-2}}\vol_{X\ud|M_d}\left((P_x)\ud\right).\nonumber
\end{gather}

Assume now that   the  inequality $\vol_{X\ud|M_d}\left((P_x)\ud\right)\ge C h_{a_d}^0({X\ud}_{|M_d}, (P_x)\ud))$
holds for some constant $C>0$ and all $d\gg 0$.
By \eqref{eq: eq1}  we obtain
$$\vol_{X|M}(L_x)\ge \vol_{X|M}(P_x)\ge C\frac{1}{d^{2q-2}}h^0_{a_d}({X^{(d)}}_{|M_d},(P_x)^{(d)})=\frac{1}{d^{2q-2}}Ch^0_{a_d}({X^{(d)}}_{|M_d}, (L_x)^{(d)}), $$
where the first inequality follows by property (i) and the final equality comes from the very definition of $P_x$ (see \S \ref{ssec: continuous-res}).
Combining \eqref{eq: derivata phi}, \eqref{eq: derivata psi} and \eqref{eq: eq1}, we obtain:
$$\psi'(x)\ge nC \phi'(x).$$
\end{rem}

\section{Proofs of the main results}\label{sec: Severi lines}

This section is devoted to proving our main results;  we use freely the notation introduced  in Section \ref{sec: Preliminaries}.
\medskip

\subsection{The first Clifford-Severi line}
\begin{proof}[Proof of Theorem \ref{thm: Severi line 1}]
Take $H$ very ample on $A$ and let $M=a^*H$. By Theorem 6.7 in \cite{BPS1} we have:
\begin{equation}\label{eq: thm6.7}
 \lambda(L)\geq \frac{2r(L,M)}{2r(L,M)-1}\,n!
\end{equation}
 where $r(L,M)=\frac{LM^{n-1}}{K_XM^{n-1}}$ is the subcanonicity index of $L$ with respect to $M$ (\ref{ssec: subcanonicity}). Since by hypotesis $\lambda (L)=n!$, we have that $r(L,M)=+\infty$ and therefore $\kappa(X)=0$ (cf. \S ~\ref{ssec: subcanonicity}).

Now, $X$ is of maximal Albanese dimension and hence, by a criterion  of Kawamata (\cite{kawamata}) we have that $X$ is birational to an abelian variety. Since $n=\dim X \leq q=\dim A$ and the map $a$ is strongly generating, we have that $q=n$ and $a$ is birational.
\end{proof}

\begin{rem}The following analogue of \eqref{eq: thm6.7} has been proved in  \cite{barja-severi} for $L$ nef:
\begin{equation}\label{eq: barja-severi}
\lambda (L)\geq \frac{2r(L)}{2r(L)-1}\, n!
\end{equation}
where $r(L)$ is the infimum of the  $r$ such that $rK_X-L$ is pseudoeffective.
Of course $r(L,M)\le  r(L)$  for the    pullback $M$ of any  very ample $H$ on $A$, so   \eqref{eq: thm6.7} is a stronger inequality than \eqref{eq: barja-severi} above.
This is a significant improvement:  when  $\lambda(L)=n!$   inequality \eqref{eq: barja-severi} tells us  only  that $K_X$ is not big,  while inequality \eqref{eq: thm6.7} gives the much stronger condition $\kappa(X)=0$, which is crucial  for   proving  Theorem \ref{thm: Severi line 1} and, as a consequence,  Theorem \ref{thm: Severi line 2}.
\end{rem}

\subsection{The second Clifford-Severi line}

As in Section \ref{sec: Preliminaries}, we fix a very ample line bundle $H$ on $A$,  let $M=a^*H$ and $L_x=L+xM$, for $x\in\R$; recall that the eventual degree $m_{L_x}$ is defined for all $x$ such that $h^0_a(X,L_x)>0$.

We need  the following result, whose proof is a slight modification of the proof of Theorem 6.9 (i) in \cite{BPS1}.
Since that proof is spread in several intermediate results, for the ease of the reader we sketch here the main steps and refer to \cite{BPS1} for further details.


\begin{prop}\label{prop: nuova} In the assumptions of  Theorem \ref{thm: Severi line 2},   assume that we have that $m_{L_x}=1$ for all  $0\ge x\in\mathbb{Q}$ such that $h^0_a(X,L_x)>0$. Then
\begin{equation}\label{eq: terza}
\vol(L)\geq \frac{5}{2}\,n!\,h^0_a(X,L).
\end{equation}
\end{prop}

\begin{proof}
First of all observe that the hypothesis about the eventual degree is stable under multiplication maps, under passing to the  continuous moving part and restricting to a general section $M_d$ (see \S \ref{sec: Preliminaries}).

 More concretely, if $x\in \mathbb{Q}$, $d\gg 0$ and $M_d$ is a smooth general  member of $|a_d^*(H)|$ then we have
\begin{equation}\label{eq: gradoeventuale1}
m_{L_x}=1 \,\, \Rightarrow \,m_{P_x}=1 \,\,\Rightarrow \, m_{P_x^{(d)}}=1\,\,\Rightarrow \, m_{{{P_x}^{(d)}_{|M_d}}}=1.
\end{equation}
We can assume $h^0_a(X,L)> 0$, since  the result is trivially true otherwise. The proof works by induction on $n$.

For $n=2$ we only need to follow the chain of implications given in \cite{BPS1}, where Proposition 5.4 (ii) implies Theorem 5.5 (ii) which finally implies Proposition 6.14 (i): ${\rm vol}(L)\geq 5 h^0_a(X,L)$.
The assumption made in \cite{BPS1} is that $a $ is of degree 1. We do not have this hypothesis here; however, the only fact used is that the linear systems $|(P_x)\ud|_{|M_d}$ (and hence  $|(P_x)\ud_{|M_d}|$) are generically injective for $d\gg0$, and this is true by (\ref{eq: gradoeventuale1}).


For the inductive step,  take $0>x\in \Q$ with $h^0_a(X,L_x)>0$ and consider its continuous moving part  $P_x$ (cf.  \S~\ref{ssec: continuous-res}).  Note that $K_X-L_x=(K_X-L)-xM$ is pseudoeffective, since  $x<0$ and $K_X-L$ is pseudoeffective. If $t$ is such that $(L_x)\up{t}$ is integral and we have chosen a continuous resolution $\eta\colon Y\to X\up{t}$ of the continuous base locus of $(L_x)\up{t}$, then it is easy to see that $K_{Y\up{t}}-(P_x)\up{t}$ also is pseudoeffective.

So, since $m_{P_x}=1$,   for $d\gg 0$  the inductive hypothesis gives:
\begin{gather}\label{eq: inductive-step}
 \vol_{X\ud|M_d}\left({(P_x)\ud}\right)=({P_x}\ud)^{n-1}M_d=\vol_{M_d}\left({(P_x)\ud}_{|M_d}\right)\geq \\ \ge \frac{5}{2}(n-1)!h^0_{a_d}(M_d,{(P_x)\ud}_{|M_d})
\ge \frac{5}{2}(n-1)!h^0_{a_d}(X^{(d)}_{|M_d},(P_x)^{(d)}), \nonumber
\end{gather}
where the first equalities  hold since $(P_x)\ud$ is nef (cf. \S~\ref{ssec: volume}).
As in  \S \ref{ssec: continuous rank}, denote by $\phi(x)=h^0_a(X,L_x) $ the continuous rank function and by $\psi(x)=\vol(L_x)$ the volume function; by  Remark \ref{rem: Preliminaries},  inequality \eqref{eq: inductive-step} gives:
$$
\psi'(x)\geq \frac{5}{2}n!D^-\phi(x)
$$
for $0\ge x\in \Q$.
Since $\psi'(x)$ is continuous and $D^-\phi(x)$ is non decreasing, the same inequality holds for $x\in\mathbb{R}$.
Now the desired inequality is obtained by taking the integral:
$$
{\vol}(L)=\int_{-\infty}^0\psi'(x)dx\geq \frac{5}{2}n!\int_{-\infty}^0D^-\phi(x)dx=\frac{5}{2}n!h^0_a(X,L).
$$
\end{proof}
\begin{rem}
The inequality (\ref{eq: terza}) of Proposition \ref{prop: nuova} is one of the new stronger Clifford--Severi inequalities proven in \cite{BPS1}. We do not know whether this inequality is sharp (i.e. whether  there are triples $(X,a,L)$, with $a$ of degree 1 such that $\lambda(L)=(5/2)n!$).
In any case, the fact that the condition $m_{L_x}=1$ for all  $0\ge x\in\mathbb{Q}$ implies  an inequality  sharper than (\ref{eq: severi2}) is a key step for proving the characterization of the triples on the second Clifford--Severi line.
\end{rem}

We are now  able to complete the proof of Theorem \ref{thm: Severi line 2}:

\begin{proof}[Proof of Theorem \ref{thm: Severi line 2}]

\noindent (i) Let us define the following real numbers:
$$x_0:=\max\{x\ |\ \vol_X(L_x)=0\}\,\,\mbox{ and }\,\,\bar x:=\max\{x\ |\ h^0_a(X,L_x)=0\}.$$
Since $h^0_a(X,L_x)>0$ implies that $L_x$ is big, we have that $\bar x\ge  x_0$.
Consider the function
 $$\nu(x):={\vol}_X(L_x)-2n!h^0_a(X,L_x).$$
 We are going to prove that $\nu(x)$ is identically zero for $x\leq 0$.
We have   $\nu(0)=0$ by assumption and  $\nu(x) \ge 0$ for $x\le 0$  by \cite[Theorem 6.9]{BPS1}.
 Hence, it suffices to show that the left derivative $D^-\nu(x)$ is  $\ge 0$  for $x< 0$.

 Using the formulae (\ref{eq: derivata phi}) and (\ref{eq: derivata psi}) for the left  derivatives  given in Section \ref{sec: Preliminaries}, we have
 \[
 D^-\nu(x)= \begin{cases}0 & x<x_0\\

 n\vol_{X|M}(L_x)-2n!\lim_{d\to \infty}\frac{1}{d^{2q-2}}h^0_{a_d}(X\ud_{|M_d}, (L_x)\ud) &x>x_0
 \end{cases}
 \]
 Observe that for $0\ge x\in\Q$   the  inequality:
  $$\vol_{X\ud|M_d}\left((P_x)\ud\right)\ge 2(n-1)! h^0_{a_d}(X\ud_{|M_d}, (P_x)\ud))$$
   holds for $d\gg 0$ by  \cite[Theorem 6.7]{BPS1}, since we can  show that $K_{X\ud}-(P_x)\ud$ is pseudoeffective arguing as in the proof of Proposition \ref{prop: nuova}.
So for rational $0\ge x>x_0$ Remark \ref{rem: Preliminaries} gives:
$$D^-\nu(x)=n(\psi'(x)-2(n-1)!\phi'(x))\ge 0.$$
Since $\psi'$ is continuous for $x\ne x_0$ and $\phi'$ is non decreasing, the above inequality actually holds for all $ x\le 0$.
We have thus proven that for all $x\le 0$ one has:
$$
 \vol_X(L_x)=2\,n!\, h^0_a(X,L_x).
 $$

Now we can apply Theorem 6.8 of \cite{BPS1} and deduce that for all $x\in \Q\cap  (\bar x,0]$  we have $m_{L_x}=1$ or $2$. Indeed, if $L_x$ is integral, this follows directly. Otherwise take $t$ such that $L_x^{(t)}$ is integral. Since the volume and the continuous rank are multiplicative, the same inequality holds for $L_x^{(t)}$ and so we have that $m_{L_x}=m_{L_x^{(t)}}=1$ or  2.

Assume that for all rational $x\leq 0$ with $h^0_a(X,L_x)>0$, we have that $m_{L_x}=1$. Then by Proposition \ref{prop: nuova} we would have that $\vol_X (L)\geq \frac{5}{2}\, n!\,h^0_a(X,L)$, a contradiction. So there exists a rational $t_0=e_0/d_0^2\leq 0$ such that $m_{L_{t_0}}=2$.

As in \S~\ref{ssec: continuous-res},  up to passing to $X\ud$ for $d\gg0$,  taking a suitable blow-up of $X^{(d)}$ and tensoring by $\alpha$ general, we can assume  that $L_{t_0}$ is integral and  $|L_{t_0}|=|P|+D$, where $D$ is effective, $|P|$ is base point free and $h^0_{a}(X,L_{t_0})=h_a^0(X,P)$.
Moreover, we may assume  that the map $\varphi\colon X\rightarrow Z$ induced by $|P|$ is the eventual map (see \S~\ref{ssec: eventual}) and has degree  $m_{L_{t_0}}=2$. Finally, up to replacing both $X$ and $Z$ by modifications, we can assume that $Z$ is smooth, so that there is a morphism $a'\colon Z\to A$ with $a=a'\circ \varphi$. Let $N\in \pic(Z)$ be such that $P=\varphi^*(N)$; we  have  $h^0_{a}(X,P)=h^0_{a'}(Z,N)$.
 Hence we have that
\begin{gather*}
 2n!h^0_{a'}(Z,N)=2n!h^0_{a}(X,L_{t_0})=\vol_X(L_{t_0})\geq \\
  \geq \vol_X(P)= 2\vol_Z(N)\geq 2n!h^0_{a'}(Z,N),
\end{gather*}
where the last inequality follows by \cite[Theorem 6.7]{BPS1} (cf. \eqref{eq: severi1}).
 So $\vol_Z(N)=n!h^0_{a'}(Z,N)$ and we can apply Theorem \ref{thm: Severi line 1} to conclude that $q=n$ and that $a'\colon Z\to A$ is birational. Hence $\deg a=2$.

\smallskip

 \noindent (ii) Let $X\overset{\eta}{\to} \overline X\overset{\bar a}{\to} A$ be the Stein factorization of $a$; since  $\overline X$ is normal and $A$ is smooth, $\bar a$ is a flat double cover and, in particular, $\overline X$ is Gorenstein and $K_{\overline X}={\bar a}^*N$ for some line bundle $N$ on $A$. It follows that $K_X= a^*N+E$, where $E$ is $\eta$-exceptional.
 Note that $N$ is ample, since otherwise $K_X$ would not be big.

  The usual projection formulae for double covers give $h^0_{\overline a}(\overline X, K_{\overline X})=h^0_{\rm Id}(A,N)$, hence for general $\alpha\in \Pic^0(A)$ the paracanonical system $|K_{\overline X}\otimes \alpha|$ is a pull back from $A$. Since the moving part of $|K_X\otimes \alpha|$ is a subsystem of $\eta^*|K_{\overline X}\otimes \alpha|$, the map given by $|K_X\otimes \alpha|$ is composed with $a$ (recall  that $h^0_a(X,K_X)=\chi(K_X)$ by generic vanishing, so the paracanonical  systems $|K_X\otimes \alpha|$ are non empty by assumption).

  One can argue exactly in the same way for the degree 2 map $a_d\colon X\ud\to A$ and show that for $\alpha$ general  the map given by $|K_{X\ud}\otimes \alpha|$ is composed with $a_d$. It follows that $m_{K_X}=2$ and the eventual map of $K_X$ is birationally equivalent to $a$.

\smallskip

\noindent (iii) Under this hypothesis, since $h^0_a(X,L)>0$, we have  $\chi(K_X)=h^0_a(X,K_X)>0$. As we have shown in  (i), there is $0>t_0\in \Q$ such that $h^0_a(X, L_{t_0})>0$ and $m_{L_{t_0}}=2$. Up to passing to an \'etale cover induced by a multiplication map,  we have an inclusion $L_{t_0}\to  L$  and   by assumption we also have an inclusion $L \to  K_X\otimes \alpha$, for some $\alpha \in\rm{Pic}^0(A)$. Since by (ii) $m_{L_{t_0}}=m_{K_X}=2$,  we can conclude that $m_L=2$ and that $a$ is also birationally equivalent to the eventual map of $L$.

 \end{proof}

\begin{proof}[Proof of Corollary \ref{cor: Severi Canonico}]
Given  a desingularization $X'\rightarrow X$,   we have $\vol_{X'}(K_{X'})=\vol_X(K_X)=K_X^n$. Since the singularities of $X$ are rational and $\omega_X$ is the dualizing sheaf of $X$, we also have that $h^0_{a'}(X',K_{X'})=\chi(\omega_{X'})=\chi(\omega_X)$ and the result follows directly from Theorem \ref{thm: Severi line 2}.
\end{proof}

Finally we consider the case of curves on the second Severi line. In contrast to the case $n\ge 2$, here  the map $a\colon X\to A$ is not always a double cover:
\begin{thm}\label{thm: curve}
Let $C$ be a smooth curve of genus $g\geq 4$, let $a\colon  C \rightarrow A$ be a strongly generating map to an abelian variety and  let $L\in \pic(C)$ be a line bundle with $h^0_a(C,L) \geq 2$.
If $\deg L\le 2g-2$ and ${\deg}L=2h^0_a(C,L)$,  then one of the following cases occurs:
\begin{enumerate}
\item $A$ is an elliptic curve, the map $a$ has degree 2,  $L=a^*N$ for some line bundle $N$ on $A$ and $|L\otimes \alpha|$ is not  birational  for general $\alpha\in \Pic^0(A)$;
\item $\deg L=2g-2$  and $|L\otimes\alpha|$ is  birational for general $\alpha\in \Pic^0(A)$.
\end{enumerate}
\end{thm}
\begin{proof}
Up to replacing $L$ by $L\otimes \alpha$ for some general $\alpha \in \Pic^0(A)$, we may assume that $h^0(C,L)=h^0_a(C,L)=:r+1$. Plugging the relation $\deg L=2h^0(C,L)$ into the Riemann-Roch formula we get $h^0(C,L)+h^0(C,K_C-L)=g-1$, hence in particular $r\le g-2$.

Assume  $r<g-2$. In this case we  can apply an inequality of Debarre and Fahlaoui \cite[Proposition 3.3]{DF} on the dimension of the abelian varieties contained in the Brill-Noether locus $W^r_d(C)$, that in our case  gives  $\dim A\le (d-2r)/2=1$. The Brill-Noether locus $W^r_{2r+2}(C)$ is a proper subset of $J^{2r+2}(C)$ hence by \cite[Cor.~3.9]{cmp} either we are in case (i),  or  $a\colon C\to A$ has degree 2, $C$ is hyperelliptic and $L =r\Delta+a^*P$, where $\Delta$ is the hyperelliptic $g^1_2$ and $P\in A$ is a point.  To exclude the latter case, we claim that a hyperelliptic and bielliptic curve has always genus $g\leq 3$.

Indeed, since the hyperelliptic involution commutes with all the automorphisms, the composition of the bielliptic and the hyperelliptic involutions gives a third involution on $C$. The  first two involutions act on  $H^0(C,\omega_C)$ with  invariant subspaces of eigenvalue -1 of dimensions $g$ and $g-1$ respectively. So the  composition induces a double cover from $C$ to a curve of genus $g-1$. By the Hurwitz formula we obtain the bound $g\leq 3$.

So we have $L=a^*N$, where $N$ is a line bundle of degree $r+1$ on $A$ and $L\otimes\alpha=a^*(N\otimes \alpha)$  for every $\alpha \in \Pic^0(A)$; since $h^0(A, N\otimes\alpha )=r+1=h^0_a(C,L)$, it follows that $|L\otimes\alpha|$ is not birational for general $\alpha$.

Assume now $r=g-2$, namely   $\deg L= 2g-2$.  If $|L\otimes\alpha|$ is   birational for general $\alpha$ in $\Pic^0(A)$, then we have  case (ii). So assume that $|L\otimes\alpha|$ is  not birational for general $\alpha$ in $\Pic^0(A)$.
Since  $g>3$,    by the  Clifford+ Theorem \cite[III.3.~ex.B.7]{ACGH} the induced map $\phi_{|L\otimes\alpha|}$  factorizes as a double cover  $\sigma$ of a curve $D_{\alpha}$ of genus $0$ or $1$, composed with a birational map $\phi_{M_{\alpha}}$ such that the moving part of $|L\otimes\alpha|$ is $\sigma^*|M_{\alpha}|$.
Since the group $\Aut(C)$ is finite, the double cover $\sigma\colon C\to D_{\alpha}$ is independent of $\alpha\in \Pic^0(A)$ general and we may write $D=D_{\alpha}$.

If $g(D)=0$, then $C$ is hyperelliptic, $\sigma$ is the hyperelliptic double cover and $L=r\Delta+F_{\alpha}$, where $\Delta$ is the $g^1_2$ and $F_{\alpha}$ is an effective divisor of degree 2. Sending  $\alpha$ to  $F_{\alpha}$ defines  a  generically injective rational map $ f\colon A\to W_2(C)$. By \cite[Cor.~3.9]{cmp} we have, as before,   $\dim A=1$,  and there is a degree 2 map $\beta\colon C\to A$ such that  $f=\beta^*$. So $C$ is both bielliptic and hyperelliptic, hence by the argument previously used it has genus at most 3, against our assumptions.

If $g(D)=1$, then  $L\otimes\alpha$ is a pull-back from $D$. Since $a$ is strongly generating  we have that $A=D$  and hence in particular $\deg a=2$ and $\dim A=1$.
\end{proof}


\begin{thebibliography}{ABCD}

\bibitem{ACGH}  E.~Arbarello, M.~Cornalba, P.A.~Griffiths, J.~Harris, {Geometry of algebraic curves. {V}ol. {I}}, {Grundlehren der Mathematischen Wissenschaften}, {\bf 267}, {Springer-Verlag, New York}, {1985}.

\bibitem{barja-severi} M.A.~Barja, {\em Generalized Clifford Severi inequality and the volume of irregular varieties}, Duke Math. J. {\bf164} (2015), no. 3, 541--568.

\bibitem{BPS} M.A.~Barja, R.~Pardini, L.~Stoppino, {\em Surfaces on the Severi line}, Journal de Math\'ematiques Pures et Appliqu\'ees, (2016), no. 5, 734--743.

\bibitem{BPSeventual} M.A.~Barja, R.~Pardini, L.~Stoppino, {\em The eventual paracanonical map of a variety of maximal Albanese dimension}, to appear in Algebraic Geometry. arXiv:math.AG/1606.03301.

\bibitem{BPS1} M.A.~Barja, R.~Pardini, L.~Stoppino {\em Linear systems on irregular varieties}, arXiv:mathAG/1606.03290.
\bibitem{BFJ} S.~Boucksom, C.~Favre, M.~Jonsson, {\em Differentiability of volumes of divisors and a problem of Teissier},  J. Algebraic Geom. {\bf 18} (2009), no. 2, 279--308.


\bibitem{cmp} C.~Ciliberto, M.~Mendes Lopes, R.~Pardini, {\em Abelian varieties in Brill--Noether loci},  Advances in Mathematics  257C  (2014), 349--364.


\bibitem{DF} O.~Debarre, R.~Fahlaoui, {\em Abelian varieties in $W^r_d$ and points of bounded degree on algebraic curves}, {Compositio Mathematica}, {\bf 88} {(1993)}, 235--249.
\bibitem{elmnp} L.~Ein, R.~Lazarsfeld, M.~Popa, M.~Mustata and M.~Nakamaye,  {\em Restricted volumes and base loci
of linear series}, Amer. J. Math. {\bf 131} (2009), 607--651.


\bibitem {GL}  M.~Green, R.~Lazarsfeld, {\em Deformation theory, generic vanishing theorems, and some conjectures of Enriques, Catanese and Beauville},  Invent. Math., \textbf{90}, 1987, 416--440.

\bibitem{kawamata} Y.~Kawamata, {\em Characterization of abelian varieties}, Compositio Math. {\bf 43} (1981), no. 2, 253--276.

\bibitem{lazarsfeld-I}  R.~Lazarsfeld,  Positivity in algebraic geometry. I,  Ergebnisse der Mathematik und ihrer Grenzgebiete. 3. Folge. 48. Springer-Verlag, Berlin, 2004.

\bibitem{lazarsfeld-II}  R.~Lazarsfeld,  Positivity in algebraic geometry. II,  Ergebnisse der Mathematik und ihrer Grenzgebiete. 3. Folge. 49. Springer-Verlag, Berlin, 2004.


\bibitem{LM}  R.~Lazarsfeld, M.~Mustata, {\em Convex bodies associated to linear series}, Ann. Sci .Ec. Norm. Sup\'er. (4) {\bf 42} (2009), no. 5, 783--835.

\bibitem{LZ2} X.~Lu, K.~Zuo, {\em On  Severi type Inequalities for Irregular Surfaces}, I.M.R.N.  2017 rnx127.
doi: 10.1093/imrn/rnx127

\bibitem {rita-severi} R.~Pardini, {\em The Severi inequality $K^2\geq 4\chi$ for surfaces of maximal Albanese dimension}, Invent. Math. {\bf 159} 3 (2005), 669--672.

\bibitem{pp} B.~Pareschi, M.~Popa, {\em GV-sheaves, Fourier-Mukai transforms and Generic Vanishing}, Amer. J. Math. {\bf 133} (1) (2011), 235--271.

\bibitem{severi} F.~Severi, {\em La serie canonica e la teoria delle serie principali di gruppi di punti sopra una superficie algebrica}, Comm. Math. Helv. {\bf 4} (1932), 268--326.
     \end{thebibliography}
     \end{document}